\documentclass[10pt]{article}
\usepackage{graphicx,caption}
\usepackage{latexsym,booktabs}
\usepackage{amsmath}
\usepackage{amssymb}
\usepackage{amsthm}
\usepackage{hyperref}

\def \Z {{\mathbb Z}}

\def \C {{\mathbb C}}

\def \A {{\mathcal A}}
\def \M {{\mathbb M}}

\def \Bl {{\mathcal B}}

\def \c {{\mathcal{CC}}}

\def \J  {{\mathbf J}}

\def \ga {{\gamma}}

\def \la {{\lambda}}

\def \ze {{\zeta}}

\newtheorem{theorem}{Theorem}[section]
\newtheorem{cor}[theorem]{Corollary}
\newtheorem{lemma}[theorem]{Lemma}

\newtheorem{definition}[theorem]{Definition}

\title{Adjacency Algebra of Unitary Cayley Graph}
\author{A.\;Satyanarayana Reddy\\(Shiv Nadar University, Dadri,
India, satya8118@gmail.com)}
\date{}
\begin{document}
\maketitle
\begin{abstract}
A few properties of unitary Cayley graphs are explored using their
eigenvalues. It is shown that the adjacency algebra of a unitary Cayley
graph is a coherent algebra. Finally, a class of
unitary Cayley graphs that are distance regular are also obtained.
\end{abstract}
{\bf{Key Words:}} Adjacency Algebra, Circulant Graph, Coherent
Algebra, Distance Regular Graph, Ramanujan's sum .\\
{\bf{AMS(2010):}} 05C25, 05C50

\section{Introduction and Preliminaries}\label{sec:intro}
Fix a positive integer $n$ and let $\M_n(\C)$ denote the algebra of
all $n \times n$ matrices over $\C$, the set of complex numbers.
Let $X$ be a simple graph/digraph on $n$ vertices. Then the adjacency
matrix of $X$, denoted by $A(X)=[a_{ij}]$ (or simply $A$), is an  $n
\times n$ matrix, where $a_{ij}$ equals $1$  when the vertices $i$ and
$j$ are adjacent ($\{i,j\}/(i,j)$ is an edge/directed edge ) in $X$ and $0$, otherwise.
The {\em adjacency algebra} of $X$, denoted $\A(X)$, is a subalgebra
of $\M_n(\C)$ and it consists of all polynomials in $A$ with
coefficients from $\C$. 

For any two vertices $u$ and $v$ of a connected graph
$X$, let
$d(u,v)$ denote the length of the shortest path from $u$ to $v$.
Then the  {\em diameter} of a connected graph $X=(V,E)$ is $\max\{d(u,v):
\; u,v \in V\}$. It is shown in Biggs~\cite{biggs} that if $X$ is
a connected graph with diameter $D$, then $D+1\leq \dim(\A(X)) \leq n$,
where $\dim(\A(X))$ is the dimension of $\A(X)$ as a vector space
over $\C$. We now state the following two results
associated with connected regular graphs.

\begin{lemma}[A. J. Hoffman~\cite{hof}]\label{lem:J}
A graph $X$ is a connected  regular graph if and only if $\J\in
\A(X)$, where $\J$ is the matrix of all $1$'s.
\end{lemma}
A connected graph is a  {\em distance regular} if for any two vertices $u$ and $v$, the
number of vertices at distance $i$ from $u$ and $j$ from $v$ depends only on
$i$, $j$, and
the distance between $u$ and $v$. By definition, these graphs are regular. A distance regular graph with diameter $2$ is called  {\em strongly
regular graph}.

\begin{lemma}[S. S. Shrikhande and
Bhagwandas~\cite{shr}]\label{lem:strong}
A regular connected graph $X$ is strongly regular if and only if it has exactly
three distinct eigenvalues.
\end{lemma}

The graph $X$ is  {\em distance transitive} if for all vertices $u,v,x,y$
of $X$ such that
$d(u,v)=d(x,y)$, then there is a $g$ in $Aut(X)$ (the automorphism group of graph
$X$) satisfying $g(u)=x$ and $g(v)=y$. Every distance transitive graph is
distance regular. To know more about distance regular and distance 
transitive graphs refer A. E. Brouwer, A. M. Cohen, A.
Neumaier~\cite{B:R}.

For two matrices $A, B \in \M_n(\C)$, their {\em  Hadamard product},
denoted $A \circ B$, is also an $n \times n$ matrix with $(A\circ
B)_{ij} = A_{ij} B_{ij}$ for $1 \le i, j \le n$. Two matrices $A$ and $B$ are
said to be {\em disjoint} if their Hadamard product is the zero matrix.
\begin{theorem}\cite{B:R}\label{thm:sym}
Let $\mathcal{M}$ be a vector space of symmetric $n\times n$
matrices. Then  $\mathcal{M}$ has a basis of mutually disjoint
$0,1$-matrices if and only if $\mathcal{M}$ is closed under Hadamard
multiplication.
\end{theorem}

 A subalgebra of $\M_n(\C)$ containing $I, \J$, where $I$ is the identity
matrix, is said to be a {\em  coherent algebra}
if it is closed under Hadamard product and conjugate
transposition. For example,  $M_n(\C)$ and $\A(K_n)$, the adjacency
algebra of complete graph $K_n$ are the largest and smallest coherent
algebras respectively. Now we will see an example of a  non-trivial
coherent algebra.

A matrix $A \in \M_n(\C)$  is said to be {\em circulant} if
$a_{ij}=a_{{1,\;j-i+1}\;(mod\;n)}$, whenever $2\le i\le n$ and $1 \le j \le n$.
From the definition, it is clear
that if $A$ is circulant, then for each $i\geq 2$ the elements of
the $i$-th row are obtained by cyclically shifting the elements of
the $(i-1)$-th row one position to the  right. So it is sufficient
to specify its first row.
Let $W_n$ be a
circulant matrix of order $n$ with $[0\; 1\; 0 \dots
0]$ as its first row . Then the following result of Davis~\cite{davis} establishes that every circulant matrix of order $n$ is a polynomial in $W_n$.

\begin{lemma}\cite{davis}\label{lem:cir}
Let $A \in \M_n(\C)$. Then $A$ is circulant if and only if it is a polynomial over $\C$ in $W_n$.
\end{lemma}
Let $DC_n$ denotes the {\em directed cycle} with $n$ vertices. 
Then it is easy to see that $W_n$ is the adjacency matrix of $DC_n$ and
$\A(DC_n)$ is a coherent algebra of dimension $n$. Further,
$\{W_n^{0},W_n^1,W_n^2,\ldots, W_n^{n-1}\}$  is the unique basis of
$\A(DC_n)$ with mutually disjoint $0,1$-matrices.

Let $X$ be a graph and $A$ be its adjacency matrix. The {\em coherent closure} of $X$, denoted by $\c(X)$, is the smallest coherent
algebra containing  $A$. A graph $X$ is said to be {\em pattern polynomial graph} if
$\A(X)=\c(X)$. For example, distance regular graphs are pattern polynomial
graphs. 
In particular, let $C_n$ be the {\em cycle graph} with $n$ vertices. Then
$A(C_n)=W_n+W_n^{n-1}$, the adjacency matrix of $C_n$. Let $D_i = W_n^i +
W_n^{n-i}$ for  $1\le i < \lfloor \frac{n}{2} \rfloor$.
For $\tau = \lfloor \frac{n}{2} \rfloor$,  \[D_{\tau} =\left\{
\begin{array}{ll}
 W_n^{\tau} &
\mbox{if} \;n\; is\; even,\\
W_n^{\tau} + W_n^{n-\tau}, & \mbox{if} \;n\; is\; odd.
\end{array}\right.
\]

The identity
$$(x^k+x^{-k})=(x+x^{-1}) (x^{k-1}+x^{1-k})-(x^{k-2}+x^{2-k})$$
enables us to establish readily by mathematical induction that
$x^k+x^{-k}$ is a monic polynomial in $x+x^{-1}$ of degree $k$ with integral
coefficients.
Consequently, $D_i$'s for $1\le i \le \tau$ are polynomials of degree
$\leq i$ in $D_1=A(C_n)$ over $\C$.
Hence $\{D_0=I,D_1,\ldots, D_{\tau}\}$ is the unique basis for $\A(C_n)$ with mutually disjoint $0,1$-matrices. Thus from Theorem~\ref{thm:sym}, $C_n$ is a pattern polynomial graph. For more examples of  pattern polynomial graphs refer~\cite{SS}.

 For a fixed positive integer $n$, let $\Z_n$ denote the set of integers
 modulo $n$. It is well known that
$\Z_n$ forms a cyclic group with respect to addition modulo $n$ and
$U_n = \{ k : 1 \le k \le n, \gcd(k,n) = 1\} \subset \Z_n$  forms a
group with respect to multiplication modulo $n$. Also, for any set $S$, if  $|S|$ denotes the cardinality of
the set $S$, then $|U_n| = \varphi(n)$, the famous {\em
Euler-totient function}. Let $\zeta_n\in\C$
denote a {\em primitive $n$-th root of unity}, {\it i.e.}, $(\zeta_n)^n=1$
and $(\zeta_n)^k\neq 1$ for any $k = 1, 2, \ldots, n-1$. Then it is
well known that the multiplicative group generated by $\zeta_n$ is
isomorphic to the additive group $\Z_n$ and the set $\{(\zeta_n)^k:
k \in U_n\}$ is the collection of all primitive $n$-th roots of
unity. 

For any divisor $d$ of $n$, let $X_d^n$ (or in short $X_d$, when the
positive integer $n$ is clear from the context) denote the Cayley
graph ${\mbox{Cay}}(\Z_n,U_d)$ that consists of the elements of
$\Z_n$ as vertices and two vertices $x,y \in \Z_n$ are adjacent (or
$\{x,y\}$ is an edge) in $X_d^n$  if $x-y(mod\; n) \in U_d$. A graph
is called  {\em circulant} if its adjacency matrix is a circulant matrix.
Since $\Z_n$ is a cyclic group, for each divisor $d$ of $n$ the Cayley graph $X_d$ is a
circulant graph. The graph $X_n$ is commonly known as the
{\em unitary Cayley graph}.

In the remaining part of this section, we provide a few results, required for
this paper. In the Section~\ref{sec:two}, we give  few properties of unitary
Cayley graphs which are derived from their eigenvalues. In Section~\ref{sec:main},
we show that every unitary Cayley graph is a  pattern polynomial graph. We also
found the values of $n$ for which  $X_n$ is a distance regular and strongly
regular graphs.

Recall that a positive integer is said to be {\em square free} if its decomposition into prime
numbers does not have any repeated factors. Let $\ga(n)=\prod\limits_{p|n}
p$ be the
square free part of $n$, where $p$ is a prime number.
\begin{lemma}[\cite{H:W}]\label{lem:mu}
Let $n$ be a positive integer. Then $\sum\limits_{k \in U_n} \ze_n^k =
\mu(n)$, where \\
\[\mu(n)=\left\{
\begin{array}{ll}
 0, &  \mbox{ if } \; n \; \mbox{ is not square free,} \\
 1, &  \mbox{ if } \; n \; \mbox{ has even number of prime factors,}\\
-1, & \mbox{ if } \; n \; \mbox{ has odd number of prime factors.}
\end{array}\right.
\]
\end{lemma}

Now we will give two results on Ramanujan's sum.

\begin{definition}
For any positive integer $n$ and non-negative integer $m$, the
{\em Ramanujan's sum},  is defined as $c_n(m) =
\sum\limits_{k \in U_n} (\ze_n^k)^m.$singular/non-singular
\end{definition}
For example $c_n(0)=\varphi(n)$ and from Lemma~\ref{lem:mu} we have 
$c_n(1)=\mu(n)$.
\begin{lemma}\cite{T:A,  mor, mot}\label{lem:ramanujan}
Fix positive integers $m$ and $n$ and let $\mu(n)$ and $c_n(m)$ be as
defined earlier. Then
\begin{enumerate}
 \item for each divisor   $d$ of $n$,
$c_n(d)=\mu(\frac{n}{d})\frac{\varphi(n)}{\varphi(\frac{n}{d})}$.
Furthermore $c_n(m)=c_n(d)$ for all $m$ for which $\gcd(m,n)=d$.
\item \label{lem:ramanujan:int} $c_n(m)=\sum_{d|\gcd(n,m)}\mu(n/d)d$. In
particular we have $c_n(m)\in \Z$.
\end{enumerate}
\end{lemma}

Now we consider a polynomial [Motose~\cite{mot}, Laszlo Toth~\cite{tot}] with coefficients as Ramanujan's  sums, namely, $R_n (x) =
c_n(0) + c_n(1)x + c_n(2)x^2 + \dots + c_n(n-1)x^{n-1}$.
Then the following theorem due to Laszlo Toth~\cite{tot}
provides a few properties of the polynomial  $R_n(x)$.

\begin{theorem} \cite{tot}\label{thm:gam}
 Let $n\geq 1$.
\begin{enumerate}
 \item The number of non-zero coefficients of $R_n(x)$ is $\ga(n)$ and the
degree of $R_n(x)$ is $n-\frac{n}{\ga(n)}$.
\item $R_n(x)$ has coefficients $\pm 1$ if and only if $n$ is square free and in
this case the number of coefficients $\pm 1$ of $R_n(x)$ is $\varphi(n)$ for $n$
is odd and is $2\varphi(n/2)$ for n is even.
\end{enumerate}
\end{theorem}

\section{{\small{Few properties of unitary Cayley graphs from their
eigenvalues}}}\label{sec:two}
In this section, we provide few properties of unitary Cayley graphs using their
eigenvalues. Most of these results are  stated in \cite{Kol, So}. There are two
tables consisting of eigenvalues, minimal polynomial and characteristic
polynomial of adjacency matrix of $X_n$.

Let $A \in M_n(\C)$ be a circulant matrix. Then from  Lemma~\ref{lem:cir}, there exists a 
 unique polynomial $p_A(x) \in \C[x]$ of degree $\leq n-1$ 
such that
$A=p_A(W_n)$. We call $p_A(x)$, the representer polynomial of $A$. Then the
following result
about circulant matrices is well known.

\begin{lemma}\label{lem:singular}
Let $A \in M_n(\C)$ be a circulant matrix with $[a_0\; a_1\dots
a_{n-1}]$ as its first row. Then
 $p_A(x)=
\sum_{i=0}^{n-1}a_{i}x^{i} \in \C[x]$ and the eigenvalues of $A$ are given by
$p_A(\ze_n^k)$ for $ k = 0, 1,
\ldots, n-1$. Further the matrix $A$ is singular
 if and only if $\deg( \gcd( p_A(x), x^n-1 ) ) > 1$.
\end{lemma}

Let us denote the adjacency matrix of $X_d$  by $A_d$.
 Then from the definition of $X_d$,  $A_d=\sum_{k\in U_d}W_n^{\frac{nk}{d}}$.
 Hence  its  representer polynomial is  $p_{A_d}(x)=\sum_{k\in
U_d}x^{\frac{nk}{d}}$ . Note the polynomial $p_{A_n}(x)$ is same as
the polynomial $\Psi_n(x)$ defined in the paper L.
Toth~\cite{tot}. Further, he proved that $\Phi_n(x)$
divides $\Psi_n(x)-\mu(n)$. Hence we have the following result.

\begin{lemma}\label{lem:Phi}
Let $n$ be a positive integer. Then $\Phi_n(x)|p_{A_n}(x)-\mu(n)$.
\end{lemma}
If $n$ is not  square  free number then, from Lemma~\ref{lem:mu},
$\mu(n)=0$. Hence from the Lemma~\ref{lem:singular},\ref{lem:Phi}, we have the following result.
And its converse is also true, and which is shown as a part in a  corollary to next theorem.
Recall a graph is said to be singular/non-singular if its adjacency matrix is singular/non-singular.
\begin{cor}
Let $n$ be a positive integer.  Then $X_n$ is non-singular graph if and
only if  $n$  is square free.
\end{cor}

\begin{theorem}\cite{Kol, So}\label{thm:eig}
Let  $\lambda_i\;(0\leq i\leq n-1)$ are
the eigenvalues of $X_n$. Then $\la_i=c_n(i)\;(0\leq i\leq n-1)$.
\end{theorem}

By fixing the following notations, we will see few applications of  above
theorem. Let $n=2^{c_0}p_1^{c_1}p_2^{c_2}\dots p_r^{c_r}$ where
$p_1<p_2\dots<p_r$ are distinct odd primes. \\Let us
denote $D^*=\{a_{i_1}a_{i_2}\cdots a_{i_t}|i_1,i_2,\ldots,i_t\in\{1,2,\ldots,
r\}\}$ where $a_j=p_j-1\;for \;1\leq
j\leq r$.  Note for  each element
$b\in D^*$ there is a number $t \;(1\leq t\leq r)$ associated with
it. By using these notations and the  results  Lemma~\ref{lem:ramanujan},
Theorem~\ref{thm:gam} and Theorem~\ref{thm:eig} we have the following tables. Recall, spectrum of a graph $X$ is denoted by $\sigma(X)=\left( \begin{array}{ccc}
      \la_1 &\dots & \la_k\\
       m_1 & \dots &m_k
     \end{array} \right),$ where $\la_1,\ldots,\la_k$ are distinct eigenvalues of $X$ and $m_1,\ldots,m_k$ are their corresponding multiplicities respectively.
\pagebreak
\begin{table}[!h]
\caption{Spectrum of unitary Cayley Graphs}
\centering
\begin{center}
\begin{tabular}{|l|l|}
\hline
$n$ & $\sigma(X_n)$(Spectrum of the graph) \\
\hline
 $p$,  $p$ is prime &
 $\left( \begin{array}{cc}
      -1 & p-1\\
       p-1 & 1
     \end{array} \right)$ \\
\hline
$p^k$,  $p$ is prime and $k>1$ & $\left( \begin{array}{ccc}
      -p^{k-1} & 0 & (p-1)p^{k-1}\\
      p-1 & p^k-p & 1
     \end{array} \right)$ \\
\hline
$2p$,  $p$ is prime & $
    \left( \begin{array}{cccc}
      -(p-1) &-1  &1 & p-1\\
      1 & p-1 & p-1 & 1
     \end{array} \right)$\\
\hline
 $pq$, where $p$ and $q$ are odd primes &

 $\left( \begin{array}{cccc}
      1 & -(p-1) &-(q-1) & \varphi(pq)\\
      \varphi(pq) & q-1 & p-1 & 1
     \end{array} \right)$\\
     \hline
  square free even number  &

 $\left( \begin{array}{cccc}
      -1 & 1 & b & -b\\
      \varphi(n) & \varphi(n) & \varphi(n)/b &\varphi(n)/b
     \end{array} \right)$ $\forall b\in D^*$\\
     \hline
 square free odd number  &
$\left( \begin{array}{cc}
      (-1)^r & [(-1)^{r+t} b]\\
      \varphi(n) & \varphi(n)/b
    \end{array} \right)$ $\forall b\in D^*$\\
     \hline
 even but not square free  &
 $\left( \begin{array}{ccccc}
      0 & -n/\ga(n) & n/\ga(n) & (n/\ga(n))b & -(n/\ga(n))b\\
      n-\ga(n)& \varphi(n) & \varphi(n) & \varphi(n)/b &\varphi(n)/b
     \end{array} \right)$ $\forall b\in D^*$\\
     \hline
 odd but not square free  &
$\left( \begin{array}{ccc}
      0 &(-1)^r(n/\ga(n)) & [(-1)^{r+t} b](n/\ga(n))\\
      n-\ga(n)& \varphi(n) & \varphi(n)/b
    \end{array} \right)$ $\forall b\in D^*$\\
\hline
\end{tabular}
\end{center}
\label{tab:eig}
\end{table}

\begin{table}[!h]
\caption{Characteristic and minimal polynomials of unitary Cayley graphs}
\centering
\begin{tabular}{|l|l|l|}
\hline
n & Minimal polynomial  & Characteristic polynomial \\
\hline
$p$,  $p$ is prime & $(x+1)(x-(p-1))$ & $(x+1)^{p-1}(x-(p-1))$ \\
\hline
$p^k$  $p$ is prime and $k>1$ & $x(x-(p-1)p^{k-1})(x+p^{k-1})$ &
$x^{p^k-p}(x-(p-1)p^{k-1})(x+p^{k-1})^{p-1}$  \\
\hline
square free and even & $(x^2-1)\prod\limits_{b\in D^*}(x^2-b^2)$ &
$(x^2-1)^{\varphi(n)}\prod\limits_{b\in D^*}(x^2-b^2)^{\varphi(n)/b}$\\
\hline
square free and odd  & $(x-(-1)^r)\prod\limits_{b\in D^*}(x-(-1)^{r+t}b)$
&$(x-(-1)^r)^{\varphi(n)}\prod\limits_{b\in
D^*}(x-(-1)^{r+t}b)^{\varphi(n)/b}$\\
\hline
not square free,  & $x(x^2-(n/\ga(n))^2)\cdot$ &
$x^{n-\ga(n)}(x^2-(n/\ga(n))^2)\cdot$ \\
even and $r>1$ & $\prod\limits_{b\in D^*}(x^2-([n/\ga(n)]b)^2)$ &
$\prod\limits_{b\in
D^*}(x^2-([n/\ga(n)]b)^2)^{\varphi(\ga(n))/b}$\\
\hline
not square free,  & $x(x-(-1)^r(n/\ga(n)))\cdot$ &
$x^{n-\ga(n)}(x^2-(n/\ga(n))^2)\cdot$ \\
odd and $r>1$ & $\prod\limits_{b\in D^*}(x-(-1)^{r+t}([n/\ga(n)]b))$ &
$\prod\limits_{b\in
D^*}(x-(-1)^{r+t}([n/\ga(n)]b))^{\varphi(\ga(n))/b}$\\
\hline

\end{tabular}
\label{tab:min}
\end{table}
As a consequence of  Theorems~\ref{thm:gam}, \ref{thm:eig},
Lemma~\ref{lem:ramanujan} and above tables, we have the following results. One
can also refer the book by Dragos M. Cvetkovic, Michael Doob $\&$ Horst
Sachs~\cite{C:D:S} for further clarification.

\begin{cor}\label{cor:eigen}
Let $n$ be a positive integer.
\begin{enumerate}
\item $X_n$ is non-singular if and only if $n$ is square free.

\item \cite{Kol}
 a)The number of non-zero eigenvalues of $X_n$ are $\ga(n)$
 or the nullity (Dimension of the null space of $A_n$) of $X_n$ is $n-\ga(n)$.\\
 b)$X_n$ is an integral graph for every $n$. Further every non-zero eigenvalue of $X_n$ is a divisor of $\varphi(n)$.\\
c)If $n$ is not square free, then none of the eigenvalues of $X_n$ is $1$ or
$-1$.
 \item Let $\tau(n)$ is the number of positive divisors of $n$. Then the degree of minimal polynomial of $X_n$
 is \[\left\{
\begin{array}{l l} 
\tau(\ga(n))+1, & \mbox{ if } \; n \; \mbox{is not square free,} \\
\tau(n), & \mbox{ if } \; n \; \mbox{is square free.}
\end{array}\right.
\]

\item The \[
\det(A(X_n))=\left\{
\begin{array}{l l}          
 0,             & \mbox{ if } \; n \; \mbox{is not square free,} \\
-1,             & \mbox{if n=2,}\\
 p-1,           &  \mbox{ if } \; n=p \; \mbox{is an odd prime,}\\
-(p-1)^2,       &  \mbox{ if } \; n=2p \; \mbox{where p is an odd prime,}\\
(p-1)^q(q-1)^p, &  \mbox{ if } \; n=pq \; \mbox{ where p ans q are odd
primes,}\\
(-1)^r\prod\limits_{b\in D^*}[(-1)^tb]^{\varphi(n)/b}, & {\mbox{ if }} \;
n \;
{\mbox{ is a square free odd number,}}\\
\prod\limits_{b\in D^*}(-1)^{\varphi(n)/b}b^{(2\varphi(n))/b}, & \mbox{ if
} \; n \;
\mbox{ is a square free even number.}
\end{array}\right.
\]
\end{enumerate}
\end{cor}

The following result gives the graph theoretic properties of $X_n$.

\begin{cor}\label{lem:graph}
Let $n$ be a positive integer.
\begin{enumerate}
\item \cite{Kol} $X_n$ bipartite graph if and only if  $n$  even number. Further
$X_n$ is  complete bipartite graph if and only if $n=2^k$ for some $k\geq 1$.\\
$X_n$ is complete graph if and only if $n$ is prime.

\item $X_n$ is strongly regular graph if and only if $n$ is a prime power.

\item $X_n$ is crown graph (the complete bipartite graph minus 1-factor) if and
only if $n=2p$ where $p$ is an odd prime.
\end{enumerate}
\end{cor}
\section{$\A(X_n)=\c(X_n)$}\label{sec:main}

In this section we will find the values of $n$ such that  $X_n$ is a distance
regular graph and we will show that $\A(X_n)=\c(X_n)\;for\;all\; n$.
The following result can be obtained from the main result (Theorem~1.2) of $\check{S}$tefko
Miklavi$\check{c}$ and Primo$\check{z}$ Poto$\check{c}$nik~\cite{ste} and from
the Corollary~\ref{lem:graph}.
\begin{theorem}
Let $X_n$ be a unitary Cayley graph. Then $X_n$ is distance regular graph if and
only if  $n$ is a prime power  or $n=2p$, where $p$ is an odd prime.
\end{theorem}
Note that in \cite{ste}, it also is shown that, every distance regular circulant graph
is distance transitive.

Before proving the next result, recall $A_n=\sum_{k\in U_n}W_n^k$, $W_n^n=I$ and
 \[ \dim(\A(X_n))=\left\{
\begin{array}{l l} \tau(\ga(n))=\tau(n), &  {\mbox{ if }} \; n \; {\mbox{is
square free}}\\
\tau(\ga(n))+1, & {\mbox{ if }} \; n \; {\mbox{ is not a square free. }}
\end{array}\right. \]
Let $\ell=\dim(\A(X_n))-1$.  Then  the set $\{I,A_n,A_n^2,\ldots ,A_n^{\ell}\}$
is a basis for $\A(X_n)$. 
Now our objective is to find another  basis for $\A(X_n)$ with mutually 
disjoint $0,1$ matrices . For that, for each divisor $x$ of $\ga(n)$, we
define
$H_x=\sum_{d|n,\ga(n/d)=x}A_d$.
Then for $n$ is even, it is easy to verify that
\begin{eqnarray}
A_n^{2s} &=& \sum_{x|\ga(n),\; x\; is\; even}b_xH_x \label{eq:st:1}\\
A_n^{2t+1} &=& \sum_{x|\ga(n),\; x\; is\; odd}b_xH_x\label{eq:st:2}
\end{eqnarray}
where $1\leq 2s, 2t+1\leq \ell,\; b_x\in \C$. If $n$ is odd, then we have 
\begin{equation}\label{eq:f}
A_n^{f} = \sum_{x|\ga(n)}b_xH_x \end{equation}
where $1\leq f \leq \ell,\; b_x\in \C$.

\begin{theorem}\label{thm:coh}
Unitary Cayley graph is a pattern polynomial graph.
\end{theorem}

\begin{proof}
By definition of adjacency algebra, $\A(X_n)$ is a matrix subalgebra of
$M_n(\C)$, $I\in\A(X_n)$ and is  
closed with respect to conjugate transposition. And by definition, $X_n$
is a connected regular graph hence from Lemma~\ref{lem:J}, $\J\in \A(X_n)$. Consequently,
it is sufficient to prove, $A(X_n)$ is closed under Hadamard product. But from
Theorem~\ref{thm:sym}, it is equivalent to showing that $A(X_n)$ has a basis of disjoint
$0,1$-matrices.

From the  Equations~(\ref{eq:st:1}), (\ref{eq:st:2}) and (\ref{eq:f}) it follows that
for  every divisor $x$ of $\ga(n)$,  $H_x\in \A(X_n)$. Hence the set
\[\left\{
\begin{array}{l l} \{A_d|d
\;\mbox{divides}\;n\}, &  {\mbox{ when n is square free, }}\\
\{I, H_{\ga(n)}-I\}\cup
\{H_x|x\; \mbox{divides}\;\ga(n),x\ne \ga(n)\},& {\mbox{ otherwise }}
\end{array}\right. \]
 forms
\textit{the} basis for $\A(X_n)$ with disjoint $0,1$-matrices. 
\end{proof}

Since $X_n$ is a pattern polynomial graph hence it is a distance polynomial
graph, walk regular graph, strongly distance-balanced graph, edge regular graph
{\it etc}... for details refer~\cite{SS}.

The following theorem is a consequence of proposition 2.1, in \cite{muz}, which
was first given in \cite{wie}.

\begin{theorem}
Let $n$ be a positive integer and $\Bl_n=\{A_d|d\;\mbox{divides}\;n\}$. Then
$L(\Bl_n)$ is a coherent subalgebra of $M_n(\C)$ of dimension
$|\Bl_n|=\tau(n)$,
where $L(S)$ is the linear span of the  set $S$.
\end{theorem}
Hence we have $\A(K_n)\subseteq \A(X_n)\subseteq L(\Bl_n)\subseteq
\A(C_n)\subseteq \A(DC_n)\subset M_n(\C)$. Also  $\A(X_n)=
L(\Bl_n)$ 
if and only if $n$ is a square free number and $\A(K_n)=\A(X_n)= L(\Bl_n)$
if and only if $n$ is a prime number.

The following result  characterizes  integral circulant graphs also
given by  Wasin So~\cite{So}.
\begin{cor}
A graph $X$ is integral circulant graph if and only if $A(X)\in L(\Bl_n)$.
\end{cor}

We now  associate an integral circulant matrix to an even arithmetical
function. Recall an arithmetical function $f(m)$ is said to be even $(mod\; n)$ if
$f(m)=f(\gcd(m,n))\;\forall m\in \Z^+$. Let $n(\geq 2)$ be fixed. Let $E_n$
denote the set of all even functions $(mod\; n)$.  The following theorem shows that
$L(\Bl_n)$ is isomorphic to $E_n$ as a vector space over $\C$.

\begin{theorem}[Pentti Haukkanen~\cite{Pen}]
The set $E_n$ forms a complex vector space under usual sum of functions and the
scalar multiplication.
The dimension of vector space $E_n$ is $\tau(n)$.
\end{theorem}
{\small{}}

\end{document}